\documentclass{article}

\usepackage{amsmath}
\usepackage{amsthm}
\usepackage{amsfonts}
\usepackage{amssymb}
\usepackage[all]{xy}

\usepackage{xr}

\usepackage{url}

\DeclareMathSymbol{\rightrightarrows}  {\mathrel}{AMSa}{"13}

\def\sk{\operatorname{sk}}

\catcode`\@=11
\def\varholim@#1#2{\mathop{\vtop{\ialign{##\crcr
 \hfil$#1\m@th\operator@font holim$\hfil\crcr
 \noalign{\nointerlineskip\kern\ex@}#2#1\crcr
 \noalign{\nointerlineskip\kern-\ex@}\crcr}}}}
\def\hocolim{\mathpalette\varholim@\rightarrowfill@} 
\def\hoinvlim{\mathpalette\varholim@\leftarrowfill@}
\catcode`\@=\active

\newtheorem{theorem}{Theorem}
\newtheorem{lemma}[theorem]{Lemma}
\newtheorem{corollary}[theorem]{Corollary}

\theoremstyle{definition}

\newtheorem{example}[theorem]{Example}

\begin{document}


\title{Complexity reduction for path categories}

\author{J.F. Jardine}

\date{July 19, 2016}

\maketitle

\section*{Introduction}

A finite cubical complex $K$ is a subobject $K \subseteq \square^{n}$ of
a standard $n$-cell in the category of cubical sets.

The object $\square^{n}$ is represented by the poset
$\mathcal{P}(\underline{n})$ of subsets of the set $\underline{n} = \{
1,2, \dots ,n \}$. This poset is an object of the box category
$\square$ that defines cubical sets (see, for example, \cite{J40}). The complex
$K$ is defined by a list of non-degenerate cells $\sigma:
\square^{k} \subseteq \square^{n}$. These cells can be identified with
poset inclusions $[A,B] \subseteq \mathcal{P}(\underline{n})$ of
intervals, where
\begin{equation*}
  [A,B] = \{ F\ \vert\ A \subseteq F \subseteq B \}
\end{equation*}
where $A \subseteq B$ are subsets of $\underline{n}$.

As such, $K$ is a list of intervals $[A,B] \subseteq
\mathcal{P}(\underline{n})$ which is closed under taking
subintervals.
\medskip

Finite cubical complexes are the higher dimensional automata of
geometric concurrency theory. In that setting, the vertices of a
cubical complex $K$ model the states of a concurrent system, and its
$k$-cells represent (where possible) the simultaneous action of $k$
processors. The cells of the ambient $n$-cell which are not in $K$
represent constraints on the system.

The main object of study associated to $K$ in this form of concurrency
theory is its collections of execution paths. These paths are the
morphisms of the {\it path category} $P(K)$.

The path category functor is now well
known --- it is also called the fundamental category and denoted by
$\tau_{1}(K)$ in the higher categories literature
\cite{Joyal-quasi-cat}.

The emphasis in concurrency theory is different, and is completely
concerned with giving exact specifications of path categories $P(K)$
in the geometric setting described above. Techniques leading to
explicit, algorithmic calculations of path categories form the subject of this
paper.
\medskip

The {\it triangulation} $\vert K \vert$ of the finite cubical complex $K$ is
a finite simplicial complex that is defined by ``putting in the
missing edges''. More explicitly,
\begin{equation*}
  \vert \square^{n} \vert = (\Delta^{1})^{\times n} = B\mathcal{P}(\underline{n}),
\end{equation*}
is the nerve of the poset $\mathcal{P}(\underline{n})$, and $\vert K
\vert$ is constructed by gluing together such objects along the
incidence relations for the cells of $K$.

The path category functor $X \mapsto P(X)$ for simplicial sets $X$ is
most succinctly defined to be the left adjoint of the nerve
functor. The path category construction for cubical sets is a
specialization of this functor, and we can write
\begin{equation*}
  P(K) := P(\vert K \vert)
\end{equation*}
for cubical complexes $K$.

In practice, the objects of $P(K)$ are the vertices of $K$, and the
morphisms are equivalence classes of paths in $1$-cells, modulo
commutativity conditions that are defined by $2$-cells.

Similarly, the
path category $P(L)$ of a finite simplicial complex $L$ has the
vertices of $L$ as objects, and has morphisms given by equivalence
classes of paths in $1$-simplices, modulo commutativity conditions
that are defined by $2$-simplices.
\medskip

There is an algorithm for computing $P(L)$ for finite simplicial
complexes $L \subseteq \Delta^{n}$ that arises from a $2$-category
$P_{2}(L)$ that is defined by the simplices of $L$, and for which
$P(L)$ is the path component category of $P_{2}(L)$ in the sense that
there is a bijection
\begin{equation*}
P(L)(x,y) \cong
\pi_{0}(P_{2}(L)(x,y))
\end{equation*}
for all vertices $x,y$. The $2$-category $P_{2}(L)$ is
defined in \cite{pathcat}.

The algorithm can be
summarized as follows:
\begin{itemize}
\item[1)] Restrict to the $2$-skeleton $\sk_{2}(L)$ of $L$. 
\item[2)] Find all paths (strings of non-degenerate $1$-simplices)
\begin{equation*}
\omega: v_{0} \xrightarrow{\sigma_{1}} v_{1} \xrightarrow{\sigma_{2}} \dots \xrightarrow{\sigma_{k}} v_{k}
\end{equation*}
in $L$.
\item[3)] Find all morphisms in the category $P_{2}(L)(v,w)$ for all vertices $v < w$ in $L$ (ordering in $\Delta^{n}$).
\item[4)] Find the sets of path components for all categories $P_{2}(L)(v,w)$.
\end{itemize}

This algorithm is the {\it path category algorithm}. It has been coded
in C and Haskell by M. Misamore --- Misamore's code is published on
github.com and hackage.haskell.org. The original test of concept was
written by G. Denham in Macaulay 2.

Except for the first step, which is due to a basic result for path
categories \cite{pathcat} that also appears in Lemma \ref{lem 1}
below, the algorithm is brute force. It works well for toy examples,
but it is easy to generate simple examples which output very large
lists of morphism sets.

\begin{example}
The ``necklace'' $L \subseteq \Delta^{40}$ be the subcomplex
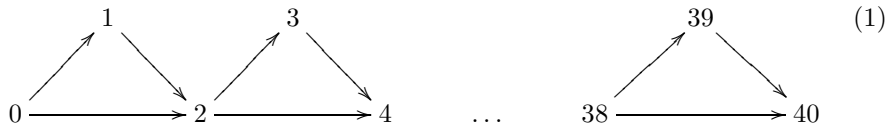
\begin{equation}\label{eq 1}
\xymatrix{
& 1 \ar[dr] && 3 \ar[dr] &&&& 39 \ar[dr] \\
0 \ar[ur] \ar[rr] && 2 \ar[ur] \ar[rr] && 4 & \dots 
& 38 \ar[ur] \ar[rr] && 40
}
\end{equation}
This is 20 copies of the complex $\partial\Delta^{2}$ glued together.
It is visually obvious that there are $2^{20}$ morphisms in
$P(L)(0,40)$, and the text file list of morphisms of $P(L)$ consumes 2 GB of
disk space.

In general, the size of the path category $P(L)$ can grow
exponentially with $L$.
\end{example}

Extreme examples aside, various complexity reduction methods have
been developed for the path category algorithm, and
the purpose of this note is to give an account of these
techniques.

The mathematical results of this paper are quite simple. Most of
the statements amount to constructions of subcomplexes $K \subseteq L$
such that the induced functor $P(K) \to P(L)$ between path categories
is fully faithful.

Explicitly,
this means that if $v,w$ are vertices of $K$, then
the induced function $P(K)(v,w) \to P(L)(v,w)$ of morphism sets is a
bijection. In this case, the morphism set $P(L)(v,w)$ can be computed
in the smaller context given by $K$, which can be much simpler
computationally.

Most of the time, $K$ is a ``full'' subcomplex of $L$. Fullness is a
general criterion for the induced functor $P(K) \to P(L)$ to be fully
faithful. The concept (appearing in Section 1 of this paper) is used
repeatedly, for the method of deletions of sources and sinks from a
simplicial complex in Section 2, and for deriving Mismore's method of
removing corners from a cubical complex in Section 3.

Section 4, on refinement of cubical complexes, is the opposite in some
sense. The idea is that one can use the data that constructs a finite
cubical complex $K$ to construct a more complicated object
$K_{\alpha}$ in a way that produces a fully faithful functor $P(K) \to
P(K_{\alpha})$. One expects that this idea will be useful for studies
of successive approximations of cubical structures.

The last section, Section 5, gives a first, coarse method for
parallelizing the path category algorithm for calculating $P(K)$ for a
cubical complex $K$. All vertices of $K$ have a size, or cardinality,
that they inherit from the ambient cell $\square^{n}$. The
resulting size functor can be used to isolate disjoint full
subcomplexes, say $A$ and $B$, for which $P(A)$ and $P(B)$ can be
computed independently. All paths $u \to v$ of $K$ which start in $A$
and end in $B$ cross a ``frontier subcomplex'' whose cells define a
coequalizer picture (see (\ref{eq 4})) that allows one to compute
$P(K)(u,v)$ from the path categories $P(A)$ and $P(B)$.

The size functor is also used in Section 4, and it is very likely to
have continuing utility. One can think of this functor as a ticking clock,
but the relationship between that ``clock'' and the higher
dimensional automaton concept can be a bit fraught.

\section{Basic results}

The first step of the path category algorithm involves a direct appeal
to the following result:

\begin{lemma}\label{lem 1}
  The inclusion $\sk_{2}(X) \subseteq X$ of the $2$-skeleton of a simplicial set $X$ induces an isomorphism of categories
  \begin{equation*}
    P(\sk_{2}(X)) \xrightarrow{\cong} P(X).
  \end{equation*}
    \end{lemma}

This result follows from the fact that the nerve $BC$ of a small
category $C$ is a $2$-coskeleton \cite[Lem. 3.5]{GJ}, which means that
there is a bijection
\begin{equation*}
  \hom(X,BC) \cong \hom(\sk_{2}(X),BC).
\end{equation*}

Lemma \ref{lem 1} is a substantial complexity reduction
step, in that it means that one can ignore much of the data for a
finite simplicial complex $L$ before computing $P(L)$.
\medskip

We now discuss a concept and result that has appeared in
connection with work on homotopy types of categories \cite[Lem. 4]{dyn}.

Suppose that $L_{0} \subseteq L$ is a subcomplex of a finite simplicial complex $L$. We say that
$L_{0}$ is a {\it
  full subcomplex} of $L$ if the following conditions hold:
\begin{itemize}
\item[1)] $L_{0}$ is path-closed in $L$, in the sense that, if there is a path 
\begin{equation*}
v=v_{0} \to v_{1} \to \dots \to
  v_{n}=v' 
\end{equation*}
in $L$ between vertices $v,v'$ of $L_{0}$, then all $v_{i} \in
  L_{0}$,
\item[2)]
if all the vertices of a simplex
$\sigma \in L$ are in $L_{0}$ then the simplex $\sigma$ is in $L_{0}$.
\end{itemize}

\begin{lemma}\label{lem 2}
Suppose that $L_{0}$ is a full subcomplex of $L$. Then the functor $P(L_{0})
\to P(L)$ is fully faithful.
\end{lemma}

Recall that a functor $F: C \to D$ is {\it fully faithful} if all
induced functions
\begin{equation*}
  f: C(x,y) \to D(f(x),f(y))
\end{equation*}
of morphism sets are bijections.

The proof of Lemma \ref{lem 2} follows from the fact that the path
category $P(L)$ is constructed by taking the category freely
associated to the graph given by the $1$-skeleton $\sk_{1}(L)$, modulo
relations defined by $2$-simplices of $L$ \cite{pathcat}. The
conditions imply that every path in $L$ between vertices $v,w$ of
$L_{0}$ consists of simplices which are in $L_{0}$, and that all
$2$-simplices which define relations of paths in $L$ between $v,w \in
L_{0}$ are also in $L_{0}$.

\begin{example}\label{ex 4}
The inclusions $d^{0}: \partial\Delta^{2} \subseteq
  \Lambda^{3}_{0}$ and $d^{3}: \partial\Delta^{2} \subseteq
  \Lambda^{3}_{3}$ induced by the respective cofaces $\Delta^{2}
  \subseteq \Delta^{3}$ both define full subcomplexes.

  In the first case, an argument on orientation says that no path in
  $\Lambda^{0}_{0}$ that starts and ends in the set of vertices $\{
  1,2,3 \}$ can pass through the vertex $0$. The second case is
  similar.
\end{example}

\begin{example}
Suppose that $i \leq j$ in $\mathbf{n}$ and suppose that $L \subseteq
\Delta^{n}$. $L[i,j]$ is the subcomplex of $L$ such that $\sigma \in
L[i,j]$ if and only if all vertices of $\sigma$ are in the interval
$[i,j]$ of vertices $v$ such that $i \leq v \leq j$. Then $L[i,j]$ is a
full subcomplex of $L$.
\end{example}

\begin{example}
Suppose that $v \leq w$ are vertices of $L \subseteq
  \Delta^{n}$. Let $L(v,w)$ be the subcomplex of $L$ consisting of
  simplices whose vertices appear on a path from $v$ to $w$. Then
  $L(v,w)$ is a full subcomplex of $L$, and of $L[v,w]$.
\end{example}

\section{Sources and sinks}

A vertex $v$ is a {\it source} of $L$ if there are no non-degenerate
$1$-simplices $u \to v$ in $L$.  The vertex $z$ is a {\it sink} of $L$
if there are no non-degenerate $1$-simplices $z \to w$ in $L$.

Every finite simplicial complex $L \subset \Delta^{n}$ has at least
one source and one sink. These are the smallest and largest vertices
of $L$, respectively, in the totally ordered set of vertices of the
ambient simplex $\Delta^{n}$.

Observe that $0$ is a source of $\Lambda^{3}_{0}$ and $3$ is a sink of
$\Lambda^{3}_{3}$. The following result formalizes the assertions made
in Example \ref{ex 4} above:

\begin{lemma}\label{lem 3}
Suppose that $S$ is a subset of the vertices of $L \subseteq \Delta^{n}$
which consists of sources and sinks. Let $L(-S)$ be the subcomplex of
$L$ which consists of simplices which do not have a vertex in
$S$. Then $L(-S)$ is a full subcomplex of $L$.
\end{lemma}

\begin{proof}
Suppose that $v < v'$ are vertices of $L(-S)$ and suppose that the
string of $1$-simplices
\begin{equation*}
v=v_{0} \to v_{1} \to \dots \to v_{n}=v'
\end{equation*}
is a path of $L$ from $v$ to $w$ consisting of non-degenerate
$1$-simplices. Then
no intermediate object $v_{i}$, $1 \leq i
\leq n-1$ can be a source or a sink. It follows that all $v_{i} \in
L(-S)$.

A simplex $\sigma$ of $L$ is in $L(-S)$ if and only if none of its
vertices are in $S$, by definition.
\end{proof}

\begin{example}
Suppose that $L$ is the complex
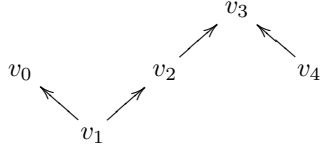
\begin{equation*}
\xymatrix@=12pt{
&&& v_{3}   \\
v_{0} && v_{2} \ar[ur]  && v_{4} \ar[ul] \\
& v_{1} \ar[ul] \ar[ur]
}
\end{equation*}
The set $S = \{v_{1},v_{3}\}$ consists of sources and sinks, and
$L(-S)$ is discrete on the vertices $v_{0},v_{2},v_{4}$. The isolated
point $v_{2}$ is a source and a sink for $L(-S)$. Let $S'= \{ v_{2}
\}$. Then
\begin{equation*}
P(L)(v_{0},v_{4}) = P(L(-S))(v_{0},v_{4}) = P(L(-S)(-S'))(v_{0},v_{4}) 
= \emptyset.
\end{equation*}

Thus, removing sources and sinks can create new
ones. The process of removing sources and sinks relative to a pair of
vertices $v,w$ of $L$ must stop, since $L$ is finite.
\end{example}

\begin{lemma}\label{lem 5}
Suppose that $v < w$ in $L$ and that $S$ consists of sources and sinks
of $L$ which are distinct from $v$ and $w$. Then
\begin{equation*}
L(-S)(v,w) = L(v,w).
\end{equation*}
\end{lemma}

\begin{proof}
Suppose that 
\begin{equation*}
\sigma: v=v_{0} \to v_{1} \to \dots \to v_{n}=w
\end{equation*}
is a path from $v$ to $w$ in $L$. Then each intermediate vertex
$v_{i}$ is neither a source or a sink, and is therefore not in $S$, so
that $v_{i} \in L(-S)$. The subcomplex $L(-S)$ is full so that the
path $\sigma$ is in $L(-S)$.

Thus, every vertex of $L(v,w)$ is a vertex of
$L(-S)(v,w)$, so that the two complexes have the same set of vertices.
These are full subcomplexes of $L$ having the same sets of vertices,
so that the inclusion
\begin{equation*}
  L(-S)(v,w) \subseteq L(v,w)
\end{equation*}
is an identity.
\end{proof}

\begin{lemma}\label{lem 6}
Suppose that $v \leq w$ in $L$, where $v$ is a source and $w$ is a
sink. Suppose given complexes
\begin{equation*}
L_{n} \subseteq L_{n-1} \subseteq \dots \subseteq L_{0}=L
\end{equation*}
where $v,w \in L_{i+1} = L_{i}(-S_{i})$ and $S_{i}$ is some set of sources and
sinks in $L_{i}$. Suppose that $L_{n}$ has a unique source $v$ and a unique
sink $w$. Then $L_{n}=L(v,w)$.
\end{lemma}

\begin{proof}
  The connected component of $v$ in $L_{n}$ has a sink, which must be
  $w$. All other components would have sources and sinks, and must therefore
  be empty. It follows that $L_{n}$ is connected.
  
 If $L_{n}$ has a vertex $x$ other than $v,w$ then there are non-degenerate $1$-simplices
\begin{equation*}
  a_{1} \to x \to b_{1}.
\end{equation*}
If $a_{1}$ is a source then $a_{1}=v$. Otherwise, there is a $1$-simplex $a_{2} \to a_{1}$. This procedure must stop, to produce a path
\begin{equation*}
  v = a_{r} \to \dots \to a_{2} \to a_{1} \to x.
\end{equation*}
Similarly, there is a path
\begin{equation*}
  x \to b_{1} \to b_{2} \to \dots \to b_{s} = w.
\end{equation*}

If $L_{n}$ has no vertices other than $v,w$, then $L_{n}$ consists of
the $1$-simplex $v \to w$.

It follows that every vertex of $L_{n}$ is on a path from $v$ to $w$,
so that $L_{n}(v,w) = L_{n}$. Then Lemma \ref{lem 5}
implies that $L_{n}(v,w) = L(v,w)$, so that $L_{n} = L(v,w)$.
\end{proof}

Suppose that $v \leq w$ in $L$, and start with
$L_{0}=L[v,w]$. Let $S_{0}$ be the set of all sources and sinks of
$L_{0}$, except for the elements $v,w$, and set $L_{1} =
L_{0}(-S_{0})$. Repeat this procedure inductively to produce a
descending chain of complexes
\begin{equation*}
  L_{n} \subseteq L_{n-1} \subseteq \dots \subseteq L_{0}=L[v,w],
\end{equation*}
with $S_{n} = \emptyset$. Then
\begin{equation*}
  L_{n} = L[v,w](v,w) = L(v,w),
\end{equation*}
by Lemma \ref{lem 6}.

In other words, starting with the full subcomplex $L[v,w]$ we can
successively delete sources and sinks to produce $L(v,w)$, which is
the minimal full subcomplex of $L$ that computes $P(L)(v,w)$.

\section{Corners}

Suppose that $i: K \subseteq \square^{n}$ is a finite cubical
complex. The inclusion $i$ induces a functor
\begin{equation*}
  i_{\ast}: P(K) \to P(\square^{n}) = \mathcal{P}(\underline{n}).
\end{equation*}
There is a poset map $t: \mathcal{P}(\underline{n}) \to \mathbb{N}$ that
is defined by cardinality, in the sense that
\begin{equation*}
  F \mapsto t(F) = \vert F \vert
  \end{equation*}
for all subsets $F$ of $\underline{n}$.
The composite functor
\begin{equation*}
P(K) \xrightarrow{i_{\ast}} \mathcal{P}(\underline{n}) \xrightarrow{t} \mathbb{N}
\end{equation*}
will also be denoted by $t$.

One thinks of the functor $t$ as a sort of time parameter for
$K$. This functor also behaves like a total degree.

\medskip

Suppose that $x$ is a vertex of the finite cubical complex $K$. Say
that $x$ is a {\it corner} if it belongs to only one maximal
cell of $K$.

The following result was proved by M. Misamore in \cite{Mis-path}. The
proof that is given here is quite different.

\begin{lemma}\label{lem 7}
  Suppose that $x$ is a corner of $K$, and let $K_{x}$ be the
  subcomplex of cells which do not have $x$ as a vertex. Then the
  functor
  \begin{equation*}
    P(K_{x}) \to P(K)
  \end{equation*}
  is fully faithful.
\end{lemma}

\begin{proof}
  Suppose that $\sigma$ is the unique top cell containing $x$.
  
  If $x$ is either maximal or minimal in $\sigma$, then $x$ is either
  a sink or a source, respectively, by the uniqueness of $\sigma$. In
  that case, the functor $P(K_{x}) \to P(K)$ is fully faithful, by
  Lemma \ref{lem 3}.
 
  Suppose that $x$ is neither maximal nor minimal in $\sigma$, and
  suppose that $P$ is a non-degenerate path in $K$ which passes
  through $x$, as in
\begin{equation*}
  P:\ u = u_{0} \to \dots \to u_{n} = v,
\end{equation*}
where $u,v \in K_{x}$, and $u_{i} = x$. Then $i \ne 0,n$, and there is
a unique $i$ such that $u_{i} = x$. In effect, since $P$ is
non-degenerate, it induces a system of proper inequalities
\begin{equation*}
\vert u \vert = \vert u_{0} \vert < \vert u_{1} \vert < \dots < \vert x \vert < \dots < \vert u_{n} \vert = \vert v \vert,
\end{equation*}
in which the number $\vert x \vert$ can only appear once.

Then
$u_{i-1}$ and $u_{i+1}$ are in $K_{x}$, and both $1$-simplices
$u_{i-1} \to x$ and $x \to u_{i+1}$ are in $\sigma$ since $\sigma$ is
the unique maximal cell that contains $x$.

Write $\sigma = [A,B]$.

As subsets of $B \subseteq \underline{n}$, $x = u_{i-1} \cup \{ a \}$ and $u_{i+1} = x \cup \{ b \}$, where $a$ and $b$ are distinct. The resulting $2$-cell
\begin{equation}\label{eq 2}
  \xymatrix{
    u_{i-1} \ar[r] \ar[d] \ar@{.>}[dr] & x \ar[d] \\
    u_{i-1} \cup \{ b \} \ar[r] & u_{i+1}
  }
  \end{equation}
in $\sigma$ (hence in $K$) defines a morphism $u_{i-1} \to u_{i}$ in $P(K_{x})$. Define $\psi(P)$ to be the composite of the morphisms
\begin{equation*}
  u = u_{0} \to \dots \to u_{i-1} \to u_{i+1} \to \dots \to u_{n} = v
\end{equation*}
in $P(K_{x})(u,v)$.

The $2$-cell of the picture (\ref{eq 2}) is uniquely determined by the path $P$, as is its image $\psi(P)$.

If $Q: u \to v$ is a non-degenerate path which
does not pass through $x$, let $\psi(Q)$ be the image of $Q$ in
$P(K_{x})(u,v)$. We have therefore determined a function
\begin{equation*}
  \psi: \{ \text{paths}\ u \to v \} \to P(K_{x})(u,v).
\end{equation*}
If there is a $2$-cell between paths $u \to v$ in $K$, then the corresponding images under $\psi$ coincide. We therefore have an induced function
\begin{equation*}
  \psi_{\ast}: P(K)(u,v) \to P(K_{x})(u,v).
\end{equation*}
The composite
\begin{equation*}
  P(K_{x})(u,v) \to P(K)(u,v) \xrightarrow{\psi_{\ast}} P(K_{x})(u,v)
\end{equation*}
is the identity by construction. The construction of $\psi(P)$ for paths $P$ passing through $x$ shows that the function
\begin{equation*}
  P(K_{x})(u,v) \to P(K)(u,v)
\end{equation*}
is surjective, and is therefore a bijection.
\end{proof}

Suppose that $x \subseteq \underline{n}$. Then $x$ is an object of the
poset $\mathcal{P}(\underline{n})$ and is a vertex of the simplicial
set $B\mathcal{P}(\underline{n})$.

Let $\square^{n}_{x}$ be the cubical subcomplex of $\square^{n}$
consisting of cells which do not have $x$ as a vertex.

Let $D_{x}$ be the subcomplex of $B\mathcal{P}(\underline{n})$
consisting of those simplices which do not have $x$ as a vertex. 
$D_{x}$ is the nerve $B\mathcal{P}(\underline{n})_{x}$ of
the full subcategory of $\mathcal{P}(\underline{n})$ with objects not
equal to $x$. In particular, the functor $P(D_{x}) \to P(\square^{n})$ is
fully faithful.

The isomorphism $\vert \square^{n} \vert \cong
B\mathcal{P}(\underline{n})$ restricts to a monomorphism of simplicial
complexes
\begin{equation*}
\gamma: \vert \square^{n}_{x} \vert \to D_{x}.
\end{equation*}

Observe that if $x$ is neither the minimal element $\emptyset$ nor maximal element
$\underline{n}$ of $\mathcal{P}(\underline{n})$, then $\emptyset
\subseteq \underline{n}$ is a $1$-simplex of $D_{x}$ which cannot be in
the image of the map $\gamma$.

If $x$ is either the maximal or minimal element of
$\mathcal{P}(\underline{n})$, then the map $\gamma$ is an
isomorphism. In effect, if $x = \underline{n}$, then a simplex $F_{0}
\subseteq \dots \subseteq F_{k}$ is in $D_{x}$ if and only if $F_{k} \ne
\underline{n}$, and in this case it is in the image of the cell $\vert
          [\emptyset,F_{k}] \vert$. The case $x = \emptyset$ is argued
          similarly.

\begin{corollary}\label{cor 8}
The functor $P(\square^{n}_{x}) \to P(\square^{n})$ is fully faithful,
and the induced functor
\begin{equation*}
\gamma_{\ast}: P(\vert \square^{n}_{x} \vert) \to P(D_{x})
\end{equation*}
is an isomorphism of path categories.
\end{corollary}

\begin{proof}
The functor
\begin{equation*}
i_{\ast}: P(\square^{n}_{x}) \to P(\square^{n})
\end{equation*}
is fully faithful by Lemma \ref{lem 7}.

  The functor $\gamma_{\ast}$ is bijective on vertices, and is also
  fully faithful by the previous paragraph. It is therefore an
  isomorphism of categories as claimed.
\end{proof}

\begin{example}
  The Swiss flag ($2$-cells indicated by double arrows, centre region is empty)
  
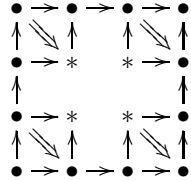
\begin{equation*}
\xymatrix@=10pt{
\bullet \ar[r] \ar@{=>}[dr]& \bullet \ar[r] & \bullet \ar[r] \ar@{=>}[dr] 
& \bullet \\
\bullet \ar[r] \ar[u] & \ast \ar[u] & \ast \ar[u] \ar[r] & \bullet \ar[u] \\
\bullet \ar[r] \ar[u] \ar@{=>}[dr] & \ast & \ast \ar[r] \ar@{=>}[dr] 
& \bullet \ar[u] \\
\bullet \ar[r] \ar[u] & \bullet \ar[r] \ar[u] & \bullet \ar[r] \ar[u] 
& \bullet \ar[u] 
}
\end{equation*}
has six corners, one sink, and one source, aside from the initial and
terminal vertices. Remove the four ``inner'' corners to show that
there are two morphisms from the initial vertex to the terminal vertex in the
corresponding path category.
\end{example}

\section{Refinement}

Suppose that $\alpha: \mathcal{P}(\underline{m}) \to
\mathcal{P}(\underline{n})$ is a poset monomorphism that preserves
meets and joins.

Every interval $[A,B]$ in $\mathcal{P}(\underline{m})$ determines an
interval $[\alpha(A),\alpha(B)]$ in $\mathcal{P}(\underline{n})$, and
$\alpha$ restricts to a poset monomorphism $\alpha: [A,B] \to
[\alpha(A),\alpha(B)]$. The assignment
\begin{equation*}
  [A,B] \mapsto [\alpha(A),\alpha(B)]
\end{equation*}
preserves inclusion
relations between intervals, and preserves meets and
joins of intervals.

The
cubical subcomplex of $\square^{n}$ that is generated by the
intervals $[\alpha(A),\alpha(B)]$ associated to the intervals $[A,B]$
of $K$ is denoted by $K_{\alpha}$, and there is a simplicial set map
$\alpha_{\ast}: \vert K \vert \to \vert K_{\alpha} \vert$ that makes
the diagram
\begin{equation*}
  \xymatrix{
    \vert K \vert \ar[d] \ar[r]^{\alpha_{\ast}} & \vert K_{\alpha} \vert \ar[d] \\
    B\mathcal{P}(\underline{m}) \ar[r]_{\alpha} & B\mathcal{P}(\underline{n})
  }
  \end{equation*}
commute. The simplicial set map $\alpha_{\ast}$ is induced by the
restricted poset morphisms $\alpha: [A,B] \to
[\alpha(A),\alpha(B)]$. These poset morphisms are not face inclusions
in general.

Suppose that $K \subseteq \square^{m}$ and $L \subseteq \square^{n}$ are
higher dimensional automata. We say that $L$ is a {\it refinement} of $K$ if
there is a poset monomorphism $\alpha: \mathcal{P}(\underline{m}) \to
\mathcal{P}(\underline{n})$ that preserves meets and joins, and an
inclusion $i: K_{\alpha} \subseteq L$ of cubical subcomplexes of
$\square^{n}$. In this case, there is a commutative diagram of
simplicial set maps
\begin{equation*}
  \xymatrix{
    \vert K \vert \ar[d] \ar[r]^{\alpha_{\ast}} & \vert K_{\alpha} \vert \ar[d] \ar[r]^{i_{\ast}}
    & \vert L \vert \ar[dl] \\
    B\mathcal{P}(\underline{m}) \ar[r]_{\alpha} & B\mathcal{P}(\underline{n})
  }
  \end{equation*}

\begin{lemma}\label{lem 4x}
Suppose that $\alpha: \mathcal{P}(\underline{m}) \to
\mathcal{P}(\underline{n})$ is a poset monomorphism which preserves meets
and joins, and suppose that $K \subseteq \square^{r}$ is a cubical
subcomplex.

Then the induced functor $\alpha_{\ast}: P(K) \to P(K_{\alpha})$ is fully faithful.
  \end{lemma}

\begin{proof}
Suppose that $F$ is a vertex of $K_{\alpha}$. Then $F \subseteq
[\alpha(A),\alpha(B)]$ for some interval $[A,B]$ of $K$, so there is a
vertex $B$ of $K$ such that $F \subseteq \alpha(B)$. There is a minimal
such $B$, call it $B_{F}$, since $\alpha$ preserves meets.

If $F=\alpha(C)$ for some $C$, then $B_{F} = C$, since $\alpha$ is a
monomorphism. In effect, $\alpha(C) \leq \alpha(B_{F}) \leq
\alpha(C)$, so $B_{F} = C$ in this case.

Suppose that
\begin{equation*}
  \omega: \alpha(A) \to F_{1} \to \dots \to F_{k} \to \alpha(B)
  \end{equation*}
is a path in $K_{\alpha}$. Then each $F_{i} \to F_{i+1}$ is in an interval $[\alpha(C_{i}),\alpha(D_{i})]$, so that the diagram of inclusions
\begin{equation*}
  \xymatrix{
    F_{i} \ar[r] \ar[d] \ar[dr] & F_{i+1} \ar[d] \\
    \alpha(B_{F_{i}}) \ar[r] & \alpha(B_{F_{i+1}})
    }
\end{equation*}
is in that same interval. The inclusion $\alpha(B_{F_{i}}) \to
\alpha(B_{F_{i+1}})$ is the image of an inclusion $B_{F_{i}} \to
B_{F_{i+1}}$ by the minimality of $B_{F_{i}}$. It follows that the
diagram
\begin{equation}\label{eq A}
  \xymatrix{
    \alpha(A) \ar[r] \ar[dr] \ar[d]_{1} & F_{1} \ar[dr] \ar[r] \ar[d] & \dots \ar[r] \ar[dr] 
    & F_{k} \ar[r] \ar[d] \ar[dr] & \alpha(B) \ar[d]^{1} \\
    \alpha(A) \ar[r] & \alpha(B_{F_{1}}) \ar[r] & \dots \ar[r]
    & \alpha(B_{F_{k}}) \ar[r] & \alpha(B)
  }
\end{equation}
defines a homotopy in $\vert K_{\alpha} \vert$ from the path $\omega$
to the path along the bottom, which path is in the image of the
function $P(K)(A,B) \to P(K_{\alpha})(\alpha(A),\alpha(B))$, because
all displayed simplices are in $\vert K_{\alpha} \vert$.

Suppose given a commutative diagram
\begin{equation*}
  \xymatrix{
    \alpha(A) \ar[r] & \dots \ar[r] & F_{i} \ar[rr] \ar[dr]
    && F_{i+1} \ar[r] & \dots \ar[r] & \alpha(B)\\
    &&& F \ar[ur]
  }
\end{equation*}
where $\omega$ is the path along the top, and the displayed triangle
of inclusions defines a $2$-simplex $\sigma$ of $\vert K_{\alpha} \vert$. This $2$-simplex is in some interval $[\alpha(C),\alpha(D)]$, and the corresponding diagram
\begin{equation*}
  \xymatrix{
    \alpha(B_{F_{i}}) \ar[rr] \ar[dr] && \alpha(B_{F_{i+1}}) \\
    & \alpha(B_{F}) \ar[ur] 
  }
  \end{equation*}
is also in the interval $[\alpha(C),\alpha(D)]$, by minimality. This
simplex is the image of a $2$-simplex of $\vert K \vert$.

We have therefore defined a function
\begin{equation*}
  s: P(K_{\alpha}(\alpha(A),\alpha(B)) \to P(K)(A,B)
\end{equation*}
such that the composite $s \cdot \alpha_{\ast}$ is the identity on $P(K)(A,B)$.
The construction of the function $s$ and the existence of the homotopies (\ref{eq A}) together imply that the function
\begin{equation*}
  \alpha_{\ast}: P(K)(A,B) \to P(K_{\alpha})(\alpha(A),\alpha(B))
\end{equation*}
is surjective. It follows that the function $\alpha_{\ast}$ is a bijection, as required.
\end{proof}

\section{Frontier subcomplex}

Suppose that $i: K \subseteq \square^{n}$ is a finite cubical
complex. Recall from Section 3 that the assignment $F \mapsto \vert F
\vert =: t(F)$ that is defined by cardinality determines a poset morphism
$\mathcal{P}(\underline{n}) \to \mathbb{N}$, and hence a composite
functor
\begin{equation*}
  t: P(K) \xrightarrow{i_{\ast}} \mathcal{P}(\underline{n}) \xrightarrow{t} \mathbb{N}.
\end{equation*}

Observe that if $v \to w$ is a non-degenerate $1$-simplex of $\vert K
\vert$, then there is a strict containment relation $v \subset w$ as
subsets of $\underline{n}$, so that $t(v) < t(w)$. A more precise version
of this statement applies to all $1$-cells $v \to w$ of $K$: $t(w) =
t(v) +1$ for such a $1$-cell.

The functor $t$ defines full subcomplexes of the complex $K$ and its
triangulation $\vert K \vert$.  In particular, if $r < s$ in
$\mathbb{N}$, let $K(r,s)$ be the subcomplex of cells whose vertices
$F$ satisfy $r \leq t(F) \leq s$, and let $\vert K \vert(r,s)$ be the
subcomplex of $\vert K \vert$ whose simplices have vertices $F$ with
$t(f)$ in the same range.

Then we have the following:

\begin{lemma}\label{lem 10}
  \begin{itemize}
    \item[1)] $\vert K \vert(r,s)$ is a full
      subcomplex of $\vert K \vert$.
    \item[2)] The canonical map
      \begin{equation*}
        \vert K(r,s) \vert \to \vert K \vert(r,s)
      \end{equation*}
      is an isomorphism of simplicial complexes.
  \end{itemize}
\end{lemma}

\begin{proof}
  For statement 1), suppose that $v,w$ are vertices of $\vert K
  \vert(r,x)$ and that
  \begin{equation*}
    v=v_{0} \to v_{1} \to \dots \to v_{n}=w
  \end{equation*}
  is a non-degenerate path from $v$ to $w$ in $\vert K \vert$. Then
  \begin{equation*}
    r \leq t(v)=t(v_{0}) < t(v_{1}) < \dots < t(v_{n}) = t(w) \leq s,
  \end{equation*}
  so that all vertices $v_{i}$ are in $\vert K \vert(r,s)$. The higher
  simplex condition for fullness of $\vert K \vert(r,s)$ is automatic
  from the definition.

The canonical inclusion of statement 2) arises from the observation
that $K(r,s)$ is a union of cells of $K$, and the induced inclusion
$\vert K(r,s) \vert \subseteq \vert K \vert$ factors through $\vert K
\vert(r,s)$.
  
  To prove statement 2), it is enough to show that the inclusion
  \begin{equation*}
        \vert K(r,s) \vert \to \vert K \vert(r,s)
      \end{equation*}
      is surjective on non-degenerate simplices. If $\sigma$ is a
      simplex of $\vert K \vert(r,s)$, it is in
      the image of the map $\vert \square^{k} \vert \to \vert K \vert$
      which is induced by a non-degenerate cell of $K$.
      The simplex $\sigma$ has the form
      \begin{equation*}
         F_{0} \leq F_{1} \leq \dots \leq F_{p}
      \end{equation*}
      with
      \begin{equation*}
      r \leq  \vert F_{0} \vert \leq \vert F_{1} \vert \leq \dots \leq \vert F_{p} \vert \leq s,
      \end{equation*}
      and it follows that the interval $[F_{0},F_{p}]$ defines a cell
      of $K(r,s)$. The simplex $\sigma$ is therefore in $\vert K(r,s)
      \vert$.
 \end{proof}

Suppose that $K \subseteq \square^{n}$ is a finite cubical complex
such that $\sk_{2}(K) = K$, and pick $M$ such that $0 < M < n$. Let $A
= K(0,M)$ and $B = K(M+1,n)$. Then $\vert A \vert$ and $\vert B \vert$
are full subcomplexes of $\vert K \vert$ by Lemma \ref{lem 10}.

Every path
\begin{equation*}
  v_{0} \to v_{1} \to \dots \to v_{k}
\end{equation*}
in $K$ has a number $r$ (which could be $-1$ or $k$) such that $v_{i} \in A$
for $i \leq r$ and $v_{i} \in B$ for $i \geq r+1$. The {\it frontier
subcomplex} $L$ is generated by $1$-cells and $2$-cells which have
vertices in $A$ and $B$.

Suppose that $u \in A$ and $v \in B$. Suppose that the $1$-cell
$\sigma: x \to y$ has $x \in A$ and $y \in B$. Then composition with
$\sigma$ defines a map
\begin{equation*}
  \sigma_{\ast}: P(A)(u,x) \times P(B)(y,v) \to P(K)(u,v).
\end{equation*}
Suppose that $\omega: \Delta^{1} \times \Delta^{1} \to K$ is defined
by $2$-simplices $\omega_{0}$ and $\omega_{1}$ such that
$d_{1}\omega_{0} = d_{1}\omega_{1}$ and $d_{2}(\omega_{0}) \in A$ and
$d_{0}(\omega_{1}) \in B$. One of the $2$-simplices $\omega_{0}$ or
$\omega_{1}$ could be degenerate.

Consider the picture:
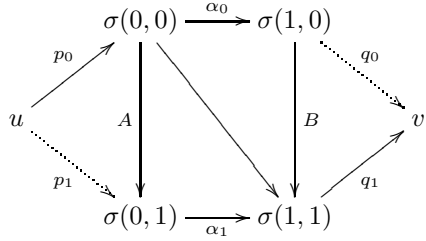
\begin{equation*}
  \xymatrix{
    & \sigma(0,0) \ar[r]^{\alpha_{0}} \ar[dd]_{A} \ar[ddr]
    & \sigma(1,0) \ar[dd]^{B} \ar@{.>}[dr]^{q_{0}} \\
    u \ar[ur]^{p_{0}} \ar@{.>}[dr]_{p_{1}} &&& v \\
    & \sigma(0,1) \ar[r]_{\alpha_{1}} & \sigma(1,1) \ar[ur]_{q_{1}}
  }
\end{equation*}
There are induced maps
\begin{equation*}
  \omega_{0}: P(A)(u,\sigma(0,0)) \times P(B)(\sigma(1,1),v) \to P(A)(u,\sigma(0,0)) \times P(B)(\sigma(1,0),v)
\end{equation*}
and
\begin{equation*}
  \omega_{0}: P(A)(u,\sigma(0,0)) \times P(B)(\sigma(1,1),v) \to P(A)(u,\sigma(0,1)) \times P(B)(\sigma(1,1),v)
  \end{equation*}

These maps define the displayed parallel pair of arrows in the
 diagram
 \begin{equation}\label{eq 4}
   \begin{aligned}
     \bigsqcup_{\omega\ \text{as above}} P(A)(u,\sigma(0,0)) \times P(B)(\sigma(1,1),v) \rightrightarrows \bigsqcup_{x \xrightarrow{\sigma} y\ \in\ L} &P(A)(u,x) \times P(B)(y,v) \\
     &\to P(K)(u,v).
   \end{aligned}
 \end{equation}

\begin{lemma}\label{lem 11}
The diagram (\ref{eq 4}) is a coequalizer.
  \end{lemma}

The proof of Lemma \ref{lem 11} is essentially by inspection.

In practical terms, Lemma \ref{lem 11} says that one can compute
$P(K)$ by first computing $P(A)$ and $P(B)$ (in parallel), and then by
stitching these calculations together with the coequalizer (\ref{eq 4}). This coequalizer defines $P(K)(u,v)$ as a set of equivalence classes on a set that we've computed, namely
\begin{equation*}
  \bigsqcup_{x \xrightarrow{\sigma} y\ \in\ L} P(A)(u,x) \times
  P(B)(y,v),
\end{equation*}
for an equivalence relation that is defined by
the parallel pair of functions in the coequalizer picture (\ref{eq 4}).
 
\nocite{pathcat}
\nocite{Joyal-quasi-cat}
\nocite{Mis-path}

\bibliographystyle{plain}
\bibliography{spt}

\end{document}